\documentclass[11pt]{article}
\usepackage{amsmath,amsfonts,amsthm,amssymb,mathrsfs,mathtools,ulem}
\usepackage{graphicx}
\usepackage[usenames,dvipsnames]{color}
\usepackage{float}
\usepackage{graphicx}
\usepackage{subfigure}
\usepackage{caption}
\usepackage{relsize,exscale}
\usepackage{cite}
\usepackage{url}
\usepackage{caption}
\usepackage{epstopdf}
\usepackage{hyperref}

\usepackage{color}

\newcommand{\dint}{\displaystyle\int}

\newcommand\redout{\bgroup\markoverwith
	{\textcolor{red}{\rule[0.5ex]{2pt}{0.8pt}}}\ULon}

\theoremstyle{plain}
\newtheorem{theorem}{Theorem}[section]

\newtheorem{corollary}[theorem]{Corollary}
\newtheorem{lemma}[theorem]{Lemma}
\newtheorem{proposition}[theorem]{Proposition}

\theoremstyle{definition}
\newtheorem{definition}[theorem]{Definition}

\theoremstyle{remark}
\newtheorem{remark}[theorem]{Remark}

\numberwithin{equation}{section}
\numberwithin{theorem}{section}

\usepackage{a4wide}
\usepackage{geometry}
\title{Hitting probabilities for fractional Brownian motion with deterministic drift}
\author{ }
\date{}

\begin{document}
	
	\maketitle
 
 \begin{center}
     
   \author{ MOHAMED ERRAOUI\\
   Department of mathematics, Faculty of science El jadida,\\ Chouaïb Doukkali University,
   Morocco \\
  e-mail\textup{: \texttt{erraoui@uca.ac.ma}}}
\end{center}

\begin{center}
         \author{ YOUSSEF HAKIKI\footnote{Supported by National Center for Scientific and Technological Research (CNRST)}\\
   Department of mathematics, Faculty of science Semlalia,\\ Cadi Ayyad University, 2390 Marrakesh, Morocco \\
  e-mail\textup{: \texttt{youssef.hakiki@ced.uca.ma}}}
 \end{center} 
	\begin{abstract}
Let $B^{H}$ be a $d$-dimensional fractional Brownian motion
with Hurst index $H\in(0,1)$, $f:[0,1]\longrightarrow\mathbb{R}^{d}$
a Borel function, and $E\subset[0,1]$, $F\subset\mathbb{R}^{d}$
are given Borel sets. The focus of this paper is on hitting probabilities of the fractional Brownian motion $B^{H}$ with the deterministic drift $f$. It aims to highlight the role of the regularity properties of the drift $f$ as well as that of the dimension of $E$ in determining the upper and lower bounds of $\mathbb{P}\{(B^H+f)(E)\cap F\neq \emptyset \}$ for $F$ a subset of $\mathbb{R}^{d}$ and also for $F$ a singleton.
\end{abstract}


\vspace{0,2cm}



	\section{Introduction}

	The hitting probability describes the probability that a given process
	will ever reach some state or set of states $F$. To find upper and lower bounds for the hitting probabilities in terms of the Hausdorff measure and the capacity of the set $F$, is a fundamental problem in probabilistic potential theory. For $d$-dimensional Brownian motion $B$ the
	probability that a path, will ever visit a given set $F\subseteq\mathbb{R}^{d}$,
	is classically estimated using the Newtonian capacity of $F$. Kakutani \cite{Kak44}
	was the first to establish this result linking capacities and hitting
	probabilities for Brownian motion. Precisely, he showed that, for $(d\geq3)$, a compact
	set $F$ is hit with positive probability by $B$ if and only if $F$ has positive Newtonian capacity.
	Since then, considerable efforts have been carried out to establish
	a series of extensions to other processes. This has given rise to
	a large and rapidly growing body of scientific literature on the subject.
	To cite a few examples, we refer to Xiao \cite{Xiao 1996} for developments on hitting probabilities of stationary
	Gaussian random fields and fractional Brownian motion; to Pruitt
	and Taylor \cite{Pruitt Taylor} and Khoshnevisan \cite{Khoshnevisan 1997} for hitting probabilities
	results for general stable processes and Lévy processes; to Khoshnevisan and Shi \cite{Khoshnevisan shi 1999} for hitting probabilities of the Brownian sheet; to Dalang and Nualart
	\cite{Dalang Nualart 2004} for hitting probabilities for the solution of a system of nonlinear
	hyperbolic stochastic partial differential equations;
	to Dalang, Khoshnevisan and Nualart \cite{Dalang Khoshnevisan Nualart 2007} and \cite{Dalang Khoshnevisan Nualart 2009},  for hitting probabilities for the solution of a 
	non-linear stochastic heat equation with additive and multiplicative noise respectively; to Xiao \cite{Xiao 2009}  Biermé,
	Lacaux and Xiao \cite{Bierme Lacaux Xiao 2009} for hitting probabilities of Gaussian random fields.  Finally,
	we refer to Khoshnevisan \cite{Khoshnevisan 2002} for more information on the latter
	as well as on potential theory of random fields.
	
	It should be noted that the above characterization is not common to
	all the processes and this is generally due to the dependence structures
	thereof leading to an upper and lower bounds on hitting probabilities
	in terms of capacity and Hausdorff measure. In this context, Chen and Xiao \cite{Chen Xiao 2012}
	improved the results established by Xiao (Theorem 7.6
	\cite{Xiao 2009}) and by Biermé, Lacaux and Xiao (Theorem 2.1 \cite{Bierme Lacaux Xiao 2009} ) on hitting probabilities of the $\mathbb{R}^{d}$-valued Gaussian
	random field $X$ satisfying conditions $(C_{1})$ and $(C_{2})$, see Xiao \cite{Xiao 2009} for precise definition, through the following

	\begin{equation}
	c^{-1}\mathcal{C}_{\rho_{H},d}(E\times F)\leq\mathbb{P}\{X(E)\cap F\neq\emptyset\}\leq c\mathcal{H}_{\rho_{H}}^{d}(E\times F),\label{Chen Xiao Estimate}
	\end{equation}
	where $E\subseteq[\varepsilon_{0},1]^{N}$, $\varepsilon_{0}\in(0,1)$
	and $F\subseteq\mathbb{R}^{d}$ are Borel sets and $c$ is a finite
	constant which depends on $[\varepsilon_{0},1]^{N}$, $F$ and $H$ only.  We emphasize that, in addition to  fractional Brownian motion, various processes are part of those satisfying conditions $(C_{1})$ and $(C_{2})$ namely, fractional Brownian sheets (Ayache and Xiao \cite{Ayache Xiao}),
	solutions to stochastic heat equation driven by space-time white
	noise (Dalang, Khoshnevisan and Nualart \cite{Dalang Khoshnevisan Nualart 2007} and \cite{Dalang Khoshnevisan Nualart 2009}, Dalang and Nualart
	\cite{Dalang Nualart 2004}), Mueller and Tribe \cite{Mueller Tribe}  and many more. See Xiao \cite{Xiao 2009}
	for more examples and further information on conditions $(C_{1})$
	and $(C_{2})$.
	$\mathcal{C}_{\rho_{H},d}$ and $\mathcal{H}_{\rho_{H}}^{d}$denotes
	the Bessel-Riesz type capacity and the Hausdorff measure with respect
	to the parabolic metric $\rho_{H}$ of order $d$. Both of these terms are defined next and will be referred as parabolic capacity and parabolic Hausdorff measure respectively.
	
	The corresponding problem for $d$-dimensional Brownian motion $B$ with drift $f$, $d\geq2$, has been considered by Peres and Souissi \cite{Peres Sousi}.
	Precisely
	they showed that for  $f:\mathbb{R}^{+}\longrightarrow\mathbb{R}^{d}$
	$(1/2)$-Hölder continuous function there exists positives constants $\alpha_{1},\alpha_{2}$
	such that for all $x\in\mathbb{R}^{d}$ and all closed set $F\subseteq\mathbb{R}^{d}$ 
	
	\[
	\alpha_{1}\text{Cap}_{M}\left(F\right)\leq\mathbb{P}_{x}\{(B+f)(0,\infty)\cap F\neq\emptyset\}\leq\alpha_{2}\text{Cap}_{M}\left(F\right),
	\]
	where $\text{Cap}_{M}\left(\cdot\right)$ denotes the Martin capacity. At the heart of their method is the strong Markov property which can't be used for fractional Brownian
	motion. Naturally, this begs the question : can we provide similar estimate
	to $\eqref{Chen Xiao Estimate}$ for $d$-dimensional fractional Brownian
	motion $B^H$ of Hurst index $H$ with drift $f$?
	
		Our first objective in this work is to give an answer to this question.
In fact, we established the desired estimates by adjusting the standard proof, which relies on the covering argument for the upper bound and the second moment argument for the lower bound, to take into account the presence of the $H$-Hölder continuous drift $f$. This allowed us to obtain, this time, according to the usual Hausdorff measure
and Bessel-Riesz capacity the upper and lower bounds on
hitting probabilities of the following type
\begin{equation*}
\mathbf{\mathfrak{c}}_{1}\,\mathcal{C}_{d-\beta_1/H}(F)\leq \mathbb{P}\{(B^H+f)(E)\cap F\neq \emptyset \}\leq \mathbf{\mathfrak{c}}_2\,\mathcal{H}^{d-\beta_2/H}(F).
\end{equation*}
Worthy of special mention is the fact that $\beta_1$ and $\beta_2$ are two different constants closely related to the Hausdorff and Minkowski dimensions  of $E$ respectively. In the event that the two dimensions coincide, often this is a consequence
of the existence of a sufficiently regular measure see condition \textbf{(S)} below, we obtain the above estimates with the same constant $\beta=\dim(E)$ the Hausdorff dimension of $E$. With these bounds in hand we draw the conclusion that $\mathbb{P}\{(B^H+f)(E)\cap F\neq \emptyset \}>0$ (resp. $\mathbb{P}\{(B^H+f)(E)\cap F\neq \emptyset \}=0$) for any compact set $F$ such that $\dim(F)>d-\beta/H$ (resp. $\dim(F)<d-\dim(E)/H$). This leads us to consider the question: is there an $\alpha$-Hölder continuous function $f$, $\alpha <H$, for which
	$\mathbb{P}\{(B^H+f)(E)\cap F\neq \emptyset \}>0$?
	
	The idea is then to take $\alpha <H$ such that $\dim(F) > d-\beta/\alpha$ and $B^\alpha$ another fractional Brownian motion with Hurst
	index $\alpha$ possibly defined on different probability space and thereafter to consider $B^H$ as a drift of $B^\alpha$. This induces us to
	bring them together on the same space while preserving their distributions.
	The best way to do this is to work on the product space and to consider
	processes on this latter. Unfortunately, we are unable to have both estimates for the same drift. These results are proved in Section 2.

	A problem related to estimating hitting probabilities for a rondom process $X$ with drift is determining which Borel functions are polar for $X$. Now we recall the definition of polar function for $X$. A Borel function $f:\mathbb{R}^{+}\longrightarrow\mathbb{R}^{d}$ is called polar for $X$ if for any $x\in \mathbb{R}^{d}$,
	\[
	\mathbb{P}\left\{ X_{t}+x=f(t)\,\text{for some }t>0\right\} =0,
	\]
	which means that the process $X-f$ does not hit points. The first
	study of polar functions for Brownian motion in $2$ dimensions appears in Graversen
	\cite{Graversen}. Precisely, he showed that for all $0<\gamma<1/2$, there
	exists a $\gamma$-Hölder continuous function $f:\mathbb{R}^{+}\longrightarrow\mathbb{R}^{2}$
	for which $B+f$ hits points. In \cite{Le Gall}, Le Gall proved that for
	any $1/2$-Hölder continuous function $f$, the process $B+f$ do
	not hits points and asked, for $d\geq3$, whether for each $\gamma<1/d$
	there exist $\gamma$-Hölder continuous functions for which $B+f$
	hits points. This problem has also been studied for the stable process by Mountford in \cite{Mountford}. Recently Antunovi$\acute{\text{c}}$, Peres and Vermesi
	\cite{Antunovic Peres Vermesi} prove first that, for $d\geq2$ and for each $\gamma<1/d$
	there exist $\gamma$-Hölder continuous functions for which  the range of $B+f$, covers an open set almost
	surely, thereby ensuring that $ B + f $ hits points. Moreover, for $d\geq3$, there exists a $1/d$-Hölder continuous
	function, accurately the $d$-dimensional Hilbert curve, such that $B+f$
	hits points. 
	Considering this problem for fractional Brownian motion is our second focus. We begin by establishing, for a general measurable drift, an upper and lower bounds on hitting probabilities as follows
	\begin{align}
	{\fontsize{14}{0} \selectfont \textbf{c}}_1^{-1}\mathcal{C}_{\rho_H,d}(Gr_E(f))\leq \mathbb{P}\left\lbrace \exists t\in E : (B^H+f)(t)=x\right\rbrace\leq {\fontsize{14}{0} \selectfont \textbf{c}}_1\, \mathcal{H}_{\rho_H}^{d}(Gr_E(f)),
	\end{align} 
	where $Gr_E(f)=\{(t,f(t)) : t\in E\}$ is  the graph of $f$ over the set $E$. The above estimates are aimed first and foremost to seek conditions on the drift $f$ for which $B^H+f$ does or does not hit points. 
	The first conclusion that we can draw is that functions with a positive parabolic capacity $\mathcal{C}_{\rho_H, d}(Gr_E (f)) $ hit points, on the other hand those who have parabolic Hausdorff measure $ \mathcal{H}_{\rho_H}^{d}(Gr_E(f))=0$ does not hit points. 
	As a first step, we prove that for any $\alpha<\dim(E)/d \wedge H$ there exists a $\alpha$-Hölder continuous function which is non-polar for $B^H$ obtained as a realization of an independent fractional Brownian motion with Hurst parameter $\alpha$. 

	The relationship between lack of regularity and fractal properties for special classes of functions has been highlighted a long time ago. Frequently graphs of continuous but sufficiently irregular functions are fractal sets what connects the lack of regularity of such functions to the Hausdorff dimension of their graphs. We consider the Weierstrass function as a prototype example of such functions. Our second step is to show that the  one dimensional fractional Brownian motion with drift given by the Weierstrass function hits points with positive probability. The two keys ingredients in the proof are a recent result of Shen \cite{Shen}, which is an improvement on the result of Barański, Bárány and Romanowska \cite{Baranski Barany Romanowska} on a long-standing conjecture concerning the Hausdorff dimension of the graph of the Weierstrass function, giving the exact value of the latter and a comparaison result for the Hausdorff parabolic dimensions with different parameters established by the authors in \cite{Erraoui Hakiki}. 
	Among the properties of Weierstrass nowhere differentiable function most often used are  $\alpha$-Hölder continuity and reverse $\alpha$-Hölder continuity for some $0 <\alpha< 1$. Przytycki and Urbański in \cite{Przytycki Urbanski}  proved that if $f : [0, 1]\rightarrow \mathbb{R}$ is both $\alpha$-Hölder and reverse $\alpha$-Hölder for some $0 < \alpha < 1$, it satisfies $\dim(Gr_{[0, 1]}(f) ) > 1$. Replacing Weierstrass function by such function $f$ we obtain the same result.
	We thought and hoped that this result continues to be true for in higher dimensions, i.e. $d\geq 2$,  but it does not. Precisely, we consider a $d$-dimensional vector-valued function $ f $ where each component is the Weierstrass function for which $\dim_{\rho_{H}}(Gr(f))<d$ leading us to conclude that  $B^H+f$ does not hit points. The above mentioned results constitute the content of Section 3. 
	
	\section{Hitting sets}
	In this section, we consider the problem on hitting probabilities of fractional Brownian motion with deterministic drift.  
	Let $H \in (0,1)$ and $B^{H}_0=\left\lbrace B^{H}_0(t),t \geq 0\right\rbrace $ be a real-valued fractional Brownian motion  of Hurst index $H$ defined on a complete  probability space $(\Omega,\mathcal{F},\mathbb{P})$, i.e. a real valued Gaussian process with stationary increments and covariance function given by $$\mathbb{E}(B^{H}_0(s)B^{H}_0(t))=\frac{1}{2}(|t|^{2H}+|s|^{2H}-|t-s|^{2H}).$$
	Let  $B^{H}_1,...,B^{H}_d$ be $d$ independent copies of $B^{H}_0$, then the stochastic process $B^{H}=\left\lbrace B^{H}(t),t\geq 0\right\rbrace $  given by  
	\[
	B^{H}(t)=(B^{H}_1(t),....,B^{H}_d(t)) ,
	\]
	is called a $d$-dimensional fractional Brownian motion of Hurst index $H \in (0,1)$. 
		
 We consider the following parabolic metric $\rho_H$ on $\mathbb{R}_+\times \mathbb{R}^d$ defined by  
	\begin{equation}
	\rho_H((s,x),(t,y))=\max\{|t-s|^H,\|x-y\|\} \quad \forall (s,x), (t,y)\in \mathbb{R}_+\times \mathbb{R}^d,\label{parabolic metric}
	\end{equation}
	where $\|.\|$ denotes the euclidean metric on $\mathbb{R}^d$. 
	For $\beta>0$ and $E\subset \mathbb{R}_+\times \mathbb{R}^d$, the $\beta$-dimensional Hausdorff measure of $E$ with respect to the metric $\rho_H$ is defined by 
	\begin{equation}
	\mathcal{H}_{\rho_H}^{\beta}(E)=\lim_{\delta \rightarrow 0}\inf \left\{\sum_{n=1}^{\infty}\left(2 r_{n}\right)^{\beta}: E \subseteq \bigcup_{n=1}^{\infty} B_{\rho_H}\left(r_{n}\right), r_{n} \leqslant \delta\right\},\label{parabolic Hausdorff measure}
	\end{equation}
	where $B_{\rho_H}(r)$ denotes an open ball of radius $r$ in the metric space $(\mathbb{R}_+\times \mathbb{R}^d,\rho_H)$. The Bessel-Riesz type capacity of order $\alpha$ on  the metric space $(\mathbb{R}_+\times\mathbb{R}^d,\rho_H)$ is defined by 
	\begin{equation}
	\mathcal{C}_{\rho_H,\alpha}(E)=\left[\inf _{\mu \in \mathcal{P}(E)} \int_{\mathbb{R}_+\times\mathbb{R}^{d}} \int_{\mathbb{R}_+\times\mathbb{R}^{d}} \varphi_{\alpha}(\rho_H(u,v)) \mu(d u) \mu(d v)\right]^{-1},\label{parabolic capacity}
	\end{equation}
	where $\mathcal{P}(E)$ is the family of probability measures carried by $E$ and the function $\varphi_{\alpha}:(0,\infty)\rightarrow (0,\infty)$ is defined by
	\begin{equation}
	\varphi_{\alpha}(r)=\left\{\begin{array}{ll}{r^{-\alpha}} & {\text { if } \alpha>0} \\ {\log \left(\frac{e}{r \wedge 1}\right)} & {\text { if } \alpha=0} \\ {1} & {\text { if } \alpha<0}\end{array}\right..\label{radial kernel}
	\end{equation}
	
	\begin{remark}
		Let $\dim_{\rho_H}$ be the Hausdorff dimension associated to the measure $\mathcal{H}_{\rho_H}^{\beta}$ which is defined as $$\dim_{\rho_H}(E)=\inf\{\beta>0: \mathcal{H}_{\rho_H}^{\beta}(E)=0\} \text{ for all $E\subset \mathbb{R}_+\times \mathbb{R}^d$},$$
		we can verify that $\dim_{\Psi, H}(.)\equiv H \times\dim_{\rho_H}(.)$, where $\dim_{\Psi, H}$ is the $H$-parabolic Hausdorff dimension which was used by Peres and Sousi in \cite{Peres&Sousi2016} in order to study the Hausdorff dimension of the graph and the image of $B^H+f$.
	\end{remark}
The usual $\beta$-dimensional Hausdorff measure, $\beta>0$, and Bessel-Riesz capacity of order $\alpha$ in Euclidean metric $\|.\|$ are denoted by $\mathcal{H}^{\beta}$ and  $\mathcal{C}_{\alpha}$ respectively. 
 $\mathcal{H}^{\beta}$ is assumed equal to $1$ whenever $\beta\leq 0$.
Let $I = [\varepsilon_0, 1]$, where $\varepsilon_0\in (0,1)$ is a fixed constant.
	First we have the following result.
	\begin{theorem}\label{main theorem hitting}
		Let $\{B^H(t), t\in [0,1]\}$ be a $d$-dimensional fractional Brownian motion and $f : [0,1] \rightarrow \mathbb{R}^d$ a Hölder continuous function with order $H$ and constant $K$. If $F\subseteq \mathbb{R}^d$ is a compact subset of $\mathbb{R}^d$ and $E$ is a Borel subset of $I$, then 
		\begin{equation}
		c_1^{-1}\mathcal{C}_{\rho_H,d}(E\times F)\leq \mathbb{P}\{(B^H+f)(E)\cap F\neq \emptyset \}\leq c_1\mathcal{H}_{\rho_H}^{d}(E\times F),\label{upper-lower bounds 1}
		\end{equation}
		where $c_1\geq 1$ is finite constant which depends on $I$, $F$, $H$ and $K$ only.
	\end{theorem}
	
	An important property of the fractional Brownian motion that will serve us well into the proof is the strong local nondeterminism which follows from Lemma 7.1 of \cite{Pitt}. Precisely, there exists a constant $0<C<\infty$ such that for all integers $n\geq 1$ and all $t_1,...,t_n,t \in [0,1]$, we have \begin{align}
	Var\left(B^H_0(t)/B^H_0(t_1),..., B^H_0(t_n)\right) \geq C \min _{0 \leq j \leq n}\left|t-t_{j}\right|^{2 H},\label{SLND FBM}
	\end{align}  
	where $Var\left(B^H_0(t)/B^H_0(t_1),..., B^H_0(t_n)\right)$ denotes the conditional variance of $B^H_0(t)$ given $B^H_0(t_1),..., B^H_0(t_n)$ and $t_0=0$.
	
	To prove the above theorem, we will make use of  the following two lemmas proved by Biermé, Lacaux and Xiao in \cite{Bierme Lacaux Xiao 2009} for a general class of Gaussian processes to which fractional Brownian motion belongs. In fact they will be used to obtain the upper and lower bounds of \eqref {upper-lower bounds 1} respectively.
	\begin{lemma}[Lemma 3.1,\cite{Bierme Lacaux Xiao 2009}]\label{l4}
		Let $\{B^H(t): t\in [0,1]\}$ be a fractional Brownian motion. For any constant $M>0$, there exist a positive constants $c_2$ and $\delta_0$ such that for all $r\in (0,\delta_0)$, $t\in I$ and all $x\in [-M,M]^d$ we have
		\begin{equation}
		\mathbb{P}\left\{\inf _{ s\in I,|s-t|^H\leq r}\|B^H(s)-x\| \leqslant r\right\} \leqslant c_2 r^{d}.\label{Gaussian estim 1}
		\end{equation}
	\end{lemma}
	
	\begin{lemma}[Lemma 3.2,\cite{Bierme Lacaux Xiao 2009} ]\label{l5}
		Let $B^H$ be a fractional Brownian motion. Then there exists a positive and finite constants $c_3$ and $c_4$ such that for all $\epsilon \in (0,1)$, $s,t\in I$ and $x,y\in \mathbb{R}^d$ we have

		\begin{align}
		& \dint_{\mathbb{R}^{2d}}e^{-i(\langle\xi,x\rangle+\langle\eta,y\rangle)}\exp\left(-\dfrac{1}{2}(\xi,\eta)\left(\varepsilon I_{2d}+\text{Cov}(B^{H}(s),B^{H}(t))\right)(\xi,\eta)^{T}\right)d\xi d\eta \nonumber \\
		\nonumber  \\
		& \leq \dfrac{(2\pi)^{d}}{\left[\text{det}\left(\Gamma_{\varepsilon}(s,t)\right)\right]^{d/2}}\exp\left(-\dfrac{c_{3}}{2}\dfrac{\|x-y\|^{2}}{\text{det}\left(\Gamma_{\varepsilon}(s,t)\right)}\right)\label{Gaussianestim2}\\
		\nonumber 	\\
		& \leq \dfrac{c_{4}}{\left(\rho_{H}((s,x),(t,y))\right)^{d}},\nonumber  
		\end{align}
		where $\Gamma_{\varepsilon}(s,t):=I_2\varepsilon+\text{Cov}(B^H_0(s),B^H_0(t))$, $I_{2d}$ and $I_2$ are the identities matrices of order $2d$ and $2$ respectively, $Cov(B^H(s),B^H(t))$ and $Cov(B^H_0(s),B^H_0(t))$ denote the covariance matrix of the random vectors $(B^H(s),B^H(t))$ and $(B^H_0(s),B^H_0(t))$ respectively, and $(\xi,\eta)^T$ is the transpose of the row vector $(\xi,\eta)$.
	\end{lemma}

	
	\begin{proof}[Proof of Theorem \ref{main theorem hitting}]
		First of all, we note that the proof of the upper bound in \eqref {upper-lower bounds 1} is similar to that of Theorem 2.1 in \cite {Chen Xiao 2012} which relies on the use of a simple covering argument.
		Since $F$ is compact set there exists a  constant $M_0>0$ such that $F\subset [-M_0,M_0]^d$. 
		Let $M_1=\sup _{s\in I}\|f(s)\|$ and $M_2=M_0+M_1$. Then for all $(t,y)\in I\times F$ we have $y+f(t) \in [-M_2,M_2]^d$. Applying Lemma \ref{l4} with the constant $M_2$ leads to the existence of a positive constants $c_5$ and $\delta_1$ such that for all $r\in (0,\delta_1)$ and  $(t,y)\in I\times F$ we have
		\begin{align}
		\mathbb{P}\left\{\inf_{s\in I,|s-t|^{H}\leq  r}\|B^H(s)-(y-f(t))\| \leq r\right\}\leq c_{5} r^d\label{proba estim},
		\end{align}
		where $c_{5}$ and $\delta_1$ depend only on $I,H$ and $M_2$.
		Now let us choose an arbitrary constant $\gamma>\mathcal{H}_{\rho_H}^{d}(E\times F)$.
		Then there is a covering of $E\times F$ by balls $\{B_{\rho_H}((t_i,y_i),r_i), i\geq 1\}$ in $\mathbb{R}_+\times\mathbb{R}^d$ such that  $r_i<\delta_1/1+K$ for all $i\geq 1$, where $K$ is the Hölder constant of $f$, and 
		\begin{equation}
		E\times F\subseteq \bigcup_{i=1}^{\infty}B_{\rho_H}((t_i,y_i),r_i)\quad  \text{with }\quad  \sum_{i=1}^{\infty}(2r_i)^{d}\leq \gamma.\label{Hausdorff covering}
		\end{equation}
		It follows that
		
		\begin{align}
		\left\{ (B^{H}+f)(E)\cap F\neq\emptyset\right\}  & =\left\{\,\exists\,(t,y)\in E\times F\,:\left(B^{H}+f\right)(t)=y\right\}  \nonumber \\
		& \subseteq\bigcup_{i=1}^{\infty}\left\{ \left(B^{H}+f\right)\left(\left(t_{i}-r_{i}^{1/H},t_{i}+r_{i}^{1/H}\right)\right)\cap B(y_{i},r_{i})\neq\emptyset\right\}. \label{event covering}
		\end{align}
		As a first step, it is easy to see that 
		$$\left\{ (B^H+f)\left(\left(t_i-r_i^{1/H},t_i+r_i^{1/H}\right)\right)\cap B(y_i,r_i) \neq \emptyset \right\}=\left\{ \inf_{ |s-t_i|^{H}< r_i}\|(B^H+f)(s)-y_i\| < r_i\right\}.$$
		On the other hand since $f$ is $H$-Hölder continuous then  for  all $s \in \left(t_i-r_i^{1/H},t_i+r_i^{1/H}\right)$ we have 
		\begin{align*}
		\left\{\|(B^H+f)(s)-y_i\|< r_i\right\}&\subset 
		\left\{\|B^H(s)-(y_i-f(t_i))\|< (1+K)r_i\right\}.
		\end{align*}
		This enables us to obtain
		\begin{align*}
		\left\{ (B^H+f)\left(\left(t_i-r_i^{1/H},t_i+r_i^{1/H}\right)\right)\cap B(y_i,r_i) \neq \emptyset \right\} \subset \left\{\inf_{ |s-t_i|^H< c_6 r_i}\|B^H(s)-(y_i-f(t_i))\| < c_6 r_i\right\}.
		\end{align*} 
		where  $c_6:=K+1>1$.
		Combining $\eqref{proba estim}$, $\eqref{Hausdorff covering}$ and  $\eqref{event covering}$ we derive that $$\mathbb{P}\{(B^H+f)(E)\cap F\neq \emptyset\}\leq c_{7}\gamma,$$ where $c_{7}$ depends only on $I,H$ and $K$. Let  $\gamma\downarrow \mathcal{H}_{\rho_h}^{d}(E\times F)$, the upper bound in $(\ref{upper-lower bounds 1})$ follows.
		
		The lower bound in \eqref{upper-lower bounds 1} can be proved by using a second moment argument. We assume that $\mathcal{C}_{\rho_H,d}(E\times F)>0$ otherwise the lower bound is obvious. We can see easily from \eqref{parabolic capacity} that there is a probability measure $\mu$ on $E\times F$ such that 
		\begin{equation}
		\mathcal{E}_{\rho_H,d}(\mu):= \int_{\mathbb{R}_+\times\mathbb{R}^{d}} \int_{\mathbb{R}_+\times\mathbb{R}^{d}} \frac{\mu(d u) \mu(d v)}{(\rho_H(u,v))^d} \leq\frac{2}{\mathcal{C}_{\rho_H,d}(E \times F)}\label{estimation energy}.
		\end{equation} 
		We consider the family of random measures $\left\lbrace \mu_n, n\geq 1\right\rbrace $  on $E\times F$ defined by
		\begin{equation}
		\begin{aligned} \int_{E\times F} g(s,x)\; \mu_{n}(d s,d x) &=\int_{E \times F} (2 \pi n)^{d / 2} \exp \left(-\frac{n\|B^H(s)+f(s)-x\|^2}{2}\right) g(s,x) \,\mu(d s,d x)\\ &=\int_{E \times F} \int_{\mathbb{R}^{d}} \exp \left(-\frac{\|\xi\|^{2}}{2 n}+i\left\langle\xi, B^H(s)+f(s)-x\right\rangle\right) g(s,x) \, d \xi \, \mu(d s,d x), \end{aligned}\label{seq random measure 1}
		\end{equation}
		thanks to the characteristic function of a Gaussian vector. Here $g$ is an arbitrary measurable function on $\mathbb{R}_+\times \mathbb{R}^d$. Our aim is  to show that  $\left\lbrace \mu_n, n\geq 1\right\rbrace $ has a subsequence which converges weakly to a finite measure $\nu$ supported on the set $\{(s,x)\in E \times F : B^H(s)+f(s)=x\}$.
		To carry out this goal, we will start by establishing the following inequalities 
		\begin{align}
		\mathbb{E}(\|\mu_n\|)\geqslant c_{8},\quad \quad \mathbb{E}(\|\mu_n\|^2)\leqslant c_{9} \mathcal{E}_{\rho_H,d}(\mu),\label{esp energy}
		\end{align}
		which constitute together with the Paley-Zygmund inequality  the cornerstone of the proof.
		Here $\|\mu_n\|$ denotes the total mass of $\mu_n$.  We emphasize that the positive constants $c_{8}$ and $c_{9}$ are independent of $\mu$ and $n$.
		By \eqref{seq random measure 1}, Fubini's theorem and the use of the characteristic function of a Gaussian vector we have 
		\begin{align}\label{th 1 claim 1}
		\begin{aligned}\mathbb{E}(\|\mu_{n}\|) & =\int_{E\times F}\int_{\mathbb{R}^{d}}e^{-i\left\langle \xi,x-f(s)\right\rangle }\exp\left(-\frac{\|\xi\|^{2}}{2n}\right)\mathbb{E}\left(e^{i\left\langle \xi,B^{H}(s)\right\rangle }\right)\,d\xi\,\mu(ds,dx)\\
		& =\int_{E\times F}\int_{\mathbb{R}^{d}}e^{-i\left\langle \xi,x-f(s)\right\rangle }\exp\left(-\frac{1}{2}\left(\frac{1}{n}+s^{2H}\right)\|\xi\|^{2}\right)\,d\xi\,\mu(ds,dx)\\
		& =\int_{E\times F}\left(\frac{2\pi}{n^{-1}+s^{2H}}\right)^{d/2}\exp\left(-\frac{\|x-f(s)\|^{2}}{2\left(n^{-1}+s^{2H}\right)}\right)\,\mu(ds,dx)\\
		& \geq\int_{E\times F}\left(\frac{2\pi}{1+s^{2H}}\right)^{d/2}\exp\left(-\frac{\|x-f(s)\|^{2}}{2s^{2H}}\right)\,\mu(ds,dx)\\
		& \geq c_{8}>0.
		\end{aligned}
		\end{align}
		Since $F$ and $f$ are bounded and $\mu$ is a probability measure we conclude that $c_{7}$ is independent of $\mu$ and $n$. This gives the first inequality in $(\ref{esp energy})$.
		
		We will now turn our attention to the second inequality  in $(\ref{esp energy})$.  By $(\ref{seq random measure 1})$ and Fubini's theorem again we obtain
		\[
		\begin{aligned}\mathbb{E}\left(\|\mu_{n}\|^{2}\right)= & \int_{E\times F}\int_{E\times F}\,\mu(ds,dx)\,\mu(dt,dy)\int_{\mathbb{R}^{2d}}e^{-i(\left\langle \xi,x-f(s)\right\rangle +\left\langle \eta,y-f(t)\right\rangle )}\\
		& \times\exp\left(-\frac{1}{2}(\xi,\eta)(n^{-1}I_{2d}+\text{Cov}(B^{H}(s),B^{H}(t)))(\xi,\eta)^{T}\right)d\xi d\eta\\
		& \leq\int_{E\times F}\int_{E\times F}\,\frac{(2\pi)^{d}}{\left[\text{det}\left(\Gamma_{1/n}(s,t)\right)\right]^{d/2}}\exp\left(-\frac{c_{3}}{2}\frac{\|x-y+f(s)-f(s)\|^{2}}{\text{det}\left(\Gamma_{1/n}(s,t)\right)}\right)\,\mu(ds,dx)\,\mu(dt,dy),
		\end{aligned}
		\]
		where the last inequality follows from Lemma $\ref{l5}$. We denote by $I_{n}\left( (s,x),(t,y)\right) $ the last integrand. Since $\|x-y+f(t)-f(s)\|\geq\vert\|x-y\|-\|f(t)-f(s)\|\vert$, we
		have that
		\[
		\begin{array}{ll}
		I_{n}\left( (s,x),(t,y)\right) \leq & \dfrac{(2\pi)^{d}}{\left[\text{det}\left(\Gamma_{1/n}(s,t)\right)\right]^{d/2}}\,\exp\left(-\dfrac{c_{3}}{2}\dfrac{\|x-y\|^{2}}{\text{det}\left(\Gamma_{1/n}(s,t)\right)}\right)\\
		\\
		& \times \exp\left(c_{3}\dfrac{\|x-y\|\|f(s)-f(t)\|}{\text{det}\left(\Gamma_{1/n}(s,t)\right)}\right)\exp\left(-\dfrac{c_{3}}{2}\dfrac{\|f(s)-f(t)\|^{2}}{\text{det}\left(\Gamma_{1/n}(s,t)\right)}\right).
		\end{array}
		\]
		Using the strong local nondeterminism property \eqref{SLND FBM} of $B_{0}^{H}$, there exists a constant $c_{10}$  such that 
		\begin{align}
		\det\Gamma_{1/n}(s,t)& \geq\det\text{Cov}\left(B_{0}^{H}(s),B_{0}^{H}(t)\right)=\text{Var}\left(B_{0}^{H}(s)\right)\,\text{Var}\left(B_{0}^{H}(t)|B_{0}^{H}(s)\right)\nonumber
		\\ &\geq c_{10}\,|t-s|^{2H}.\label{det-determinism}
		\end{align}
		Since  $f$ is $H$-Hölder continuous we have $\dfrac{\|f(t)-f(s)\|}{\sqrt{\text{det}\left(\Gamma_{1/n}(s,t)\right)}}\leq\dfrac{K}{\sqrt{c_{10}}}$ for all $s\neq t \in E$. It follows that 
		\[
		I_{n}\left( (s,x),(t,y)\right)  \leq\frac{(2\pi)^{d}}{\left[\text{det}\left(\Gamma_{1/n}(s,t)\right)\right]^{d/2}}\exp\left(-\frac{c_{3}}{2}\frac{\|x-y\|^{2}}{\text{det}\left(\Gamma_{1/n}(s,t)\right)}\right)\exp\left(\dfrac{c_{3}\,K}{\sqrt{c_{10}}}\dfrac{\|x-y\|}{\sqrt{\text{det}\left(\Gamma_{1/n}(s,t)\right)}}\right),
		\]
		which implies that 
		\[
		I_{n}\left( (s,x),(t,y)\right) \leq\frac{c_{11}\,(2\pi)^{d}}{\left[\text{det}\left(\Gamma_{1/n}(s,t)\right)\right]^{d/2}}\exp\left(-\frac{c_{3}}{4}\frac{\|x-y\|^{2}}{\text{det}\left(\Gamma_{1/n}(s,t)\right)}\right)
		\]
		where $c_{11}$ is a positive constant such that $\underset{x\geq0}{\sup} \, \exp \left( -c_{3}\dfrac{x^{2}}{4}+\dfrac{\,c_{3}\,K}{\sqrt{c_{10}}}\,x \right) \leq c_{11}$.

		\noindent It is now straightforward to deduce that there exists a positive constant  $c_{9}$ such that
		
		\begin{equation}
		I_{n}\left( (s,x),(t,y)\right) \leq \frac{c_{9}}{\rho_H\left( (s,x),(t,y)\right)^d}=\frac{c_{9}}{\max\{|s-t|^{Hd},\|x-y\|^d\}}.\label{In-estim}
		\end{equation} 
		Indeed, if $\det \Gamma_{1/n}(s,t)\geq \|x-y\|^{2}$ we have 
		\begin{align}
		\begin{aligned}
		I_{n}((s,x),(t,y))\leq \frac{(2\pi)^{d}}{(\det \Gamma_{1/n}(s,t))^{d/2}}\leq \frac{c_{9}}{|t-s|^{Hd}},\label{1st estimation}
		\end{aligned}
		\end{align}
		where we use \ref{det-determinism} for the last the inequality. 
		Otherwise, if $\det \Gamma_{1/n}(s,t)< \|x-y\|^{2}$ the elementary inequality
		$\underset{x>0}{\sup}\,x^{d/2}\,e^{-c_{3}x/4}\leq c_{12}$ enables us to obtain 
		\begin{equation}
		I_{n}((s,x),(t,y))\leq\frac{c_{9}}{\|x-y\|^{d}}.\label{2nd estimation}
		\end{equation}
		Combining \eqref{1st estimation} and \eqref{2nd estimation} leads to \eqref{In-estim}.
		Hence the second inequalities in \eqref{esp energy} follows immediately.
		
		Plugging the moment estimates of \eqref{esp energy} into the Paley–Zygmund inequality (c.f. Kahane \cite{Kahane}, p.8), allows us to confirm that $\{\mu_n, n\geq 1\}$ has a subsequence that converges weakly to a finite measure $\tilde{\mu}$  supported on the set $\{(s,x)\in E \times F : B^H(s)+f(s)=x\}$, positive with positive probability and also satisfying the moment estimates of \eqref{esp energy}. Consequently,      
		
		$$\mathbb{P}\{(B^H+f)(E) \cap F \neq \varnothing\} \geqslant \mathbb{P}\{\|\tilde{\mu}\|>0\} \geqslant \frac{[\mathbb{E}(\|\tilde{\mu}\|)]^{2}}{\mathbb{E}\left[\|\tilde{\mu}\|^{2}\right]} \geqslant c_{13} \mathcal{C}_{\rho_H, d}(E \times F),$$
		where $c_{13}=\dfrac{c_{8}^2}{c_{9}}$. All that remains to be done is take $c_{1}=c_{7}\vee 1/c_{13}$. Thus the lower and upper bounds in \eqref{upper-lower bounds 1} will follow immediately which completes the proof. 
	\end{proof}
	
	Recall that in the precise case where $E$ is an interval, Corollary 2.2. in \cite{Chen Xiao 2012} ensures that there exists a finite constant  $ c\geq 1$ depending only on $E$, $F$ and $H$ such that 
	$$c^{-1}\,\mathcal{C}_{d-1/H}(F)\leq \mathbb{P}\left\{ B^H(E)\cap F\neq \emptyset \right\}\leq c \,\mathcal{H}^{d-1/H}(F),$$ for any Borel set $F\subseteq \mathbb{R}^d$. 
	Our next goal is to establish such estimates for $(B^H+f)$ and any Borel set $E$ by means of its Hausdorff measure. However, to achieve our stated goal, we need to make use of the Minkowski dimension as well.
		We introduce now the Minkowski dimension of $E\subset [0,1]$. Let $N(E,r)$ be the smallest number of open intervals of length $r$ required to cover $E$.  The lower and upper Minkowski dimensions of $E$ are respectively defined as 
	
	\begin{equation*}
	\begin{aligned}
	&\underline{\dim}_M(E):= \liminf_{r\rightarrow 0^+}\frac{\log N(E,r)}{\log(1/r)},\\
	&\overline{\dim}_M(E):= \limsup_{r\rightarrow 0^+}\frac{\log N(E,r)}{\log(1/r)}.
	\end{aligned}       
	\end{equation*}
	Equivalently, the upper Minkowski dimension of $E$ can be written as 
	\begin{equation}
	\overline{\dim}_M(E)=\inf\{\gamma : \exists C<\infty\, \text{ such that } N(E,r)\leqslant Cr^{-\gamma} \text{ for all } r>0 \}.\label{upper Minkowski}
	\end{equation}

	\begin{proposition}\label{corollary hitting}
		Let $B^H$, $f$ and $F$ as in Theorem \ref{main theorem hitting}. Let $E$ be a subset of $I$  such that $\dim(E)>0$. Then for any  $0<\beta_1<\dim(E)\leq \overline{\dim}_M(E)<\beta_2<Hd$, we have 
		\begin{equation}
		\mathbf{\mathfrak{c}}_{1}\,\mathcal{C}_{d-\beta_1/H}(F)\leq \mathbb{P}\{(B^H+f)(E)\cap F\neq \emptyset \}\leq \mathbf{\mathfrak{c}}_2\,\mathcal{H}^{d-\beta_2/H}(F) ,\label{lower bound 2}
		\end{equation} 
		where $\mathbf{\mathfrak{c}}_{1}$ and $\mathbf{\mathfrak{c}}_{2}$ are two positive constants which depend on $E$, $F$, $H$, $K$, $\beta_1$ and $\beta_2$. 
	\end{proposition}
	
	We need the following lemma to establish the lower bound in \eqref{lower bound 2}.
	
	
	\begin{lemma}\label{estim Frostman lemma}
		Let $H\in (0,1)$, $d\geq 1$ and $0<\beta \leq Hd$. Let $\nu$  a Borel probability measure on $[0,1]$ such that, for all $ a\in [0,1]$ and $\delta>0$
		\begin{align}
		\nu([a,a+\delta])\leq \mathbf{\mathfrak{c}}_{3} \,\delta^{\beta},\label{Frostman}
		\end{align}
	 where $\mathbf{\mathfrak{c}}_{3}$ is a positive constant which depends on $\beta$ only. Then, for all $r> 0$, we have 
		\begin{equation}
		\sup_{t\in [0,1]}\int_{[0,1]}\frac{\nu(ds)}{\max\{r^d,\vert s-t \vert^{Hd}\}}\leq \mathbf{\mathfrak{c}}_{4} \varphi_{d-\beta/H}(r),\label{estim Frostman}
		\end{equation}
		where $\varphi_{d-\beta/H}(.)$ is the function defined in \eqref{radial kernel} and $\mathbf{\mathfrak{c}}_{4}$ is a positive constant depending only on $\beta$, $H$ and $d$. \\
	\end{lemma}
	
	\begin{proof}
		First, it is worthwhile pointing out that for $r\geq 1$	 we have 
		\begin{equation}
		\sup_{t\in [0,1]}\int_{[0,1]}\frac{\nu(ds)}{\max\{r^d,\vert s-t \vert^{Hd}\}}\leq r^{-d}\leq \varphi_{d-\beta/H}(r).\label{est-I1}
		\end{equation}
		Now we assume that $0<\beta < Hd$. For $r\in (0,1)$, we divide the integral in \eqref{estim Frostman} into two parts $I_1$, $I_2$ as follows
		$$ I_1=\int_{\vert s-t \vert<r^{1/H}}\frac{\nu(ds)}{r^d} \text{ \,and \,} I_2= \int_{\vert s-t \vert\geq r^{1/H}}\frac{\nu(ds)}{\vert s-t \vert^{Hd}}.$$
		By using $\eqref{Frostman}$ we obtain 
		\begin{equation}
		I_1\leq \mathbf{\mathfrak{c}}_{3}\, 2^\beta \, \varphi_{d-\beta/H}(r).\label{estim I_1}
		\end{equation} 
		Let us set  $k(r):=\min\{k\geq 0: 2^{-k}\leq r^{1/H}\}$. Then it is easy to see that 
		\begin{align}
		[r^{1/H},1)\subset \bigcup_{k=1}^{k(r)}[2^{-k},2^{-k+1}) .\label{estim k(r)}
		\end{align}
		A second use of \eqref{Frostman} gives 
		\begin{align}
		I_2&\leq \sum_{k=1}^{k(r)}2^{kH\,d}\nu(\{s\in [0,1]: 2^{-k}\leq \vert s-t \vert <2^{-k+1} \})\nonumber\\
		&\leq 2\,\mathbf{\mathfrak{c}}_{3} \sum_{k=1}^{k(r)}2^{k(Hd-\beta)} \leq 2\,\mathbf{\mathfrak{c}}_{3} \,\dfrac{2^{(Hd-\beta)}}{2^{(Hd-\beta)}-1}  r^{-(d-\beta/H)}=2\,\mathbf{\mathfrak{c}}_{3} \,\dfrac{2^{(Hd-\beta)}}{2^{(Hd-\beta)}-1} \,\varphi_{d-\beta/H}(r). \quad \quad \quad \label{estim I_2}
		\end{align}
		Finally, putting it all together enables us to deduce \eqref{estim Frostman}.
		
		For $\beta=H\,d$ the same techniques as above can give that 
		\begin{equation*}
		I_1\leq  2^\beta\, \mathfrak{c}_{3}\, \quad \text{ and }\quad I_2\leq 2\,\mathfrak{c}_{3} \,k(r).
		\end{equation*}
		It follows from the definition of $k(r)$ that $r^{1/H}<2^{-k(r)+1}$. Hence, we have $$	I_1\leq  2^\beta\, \mathfrak{c}_{3}\, \varphi_{0}(r)\, \quad \text{ and }\quad I_2\leq 2\,\mathfrak{c}_{3} (1\vee 1/H\log(2))\,\varphi_{0}(r),$$
	which ends the proof.
		
	\end{proof}
	
	
	\begin{proof}[Proof of Proposition \ref{corollary hitting}]
		Using Theorem \ref{main theorem hitting} it suffices to prove that that there exists a positive constant $\mathbf{\mathfrak{c}}_{6}$ such that
		\begin{align}
		\mathcal{C}_{d-\beta_1/H}(F)\leq \mathbf{\mathfrak{c}}_{5}\,\mathcal{C}_{\rho_H,d}(E\times F)\quad \text{ and }\quad  \mathcal{H}_{\rho_H}^d(E\times F)\leq\, \mathbf{\mathfrak{c}}_{6} \mathcal{H}^{d-\beta_2/H}(F).\label{estim capacity-measure}
		\end{align}
		Indeed for $\beta_1 \in (0,\dim(E))$, by Frostman's theorem there is a Borel probability measure $\nu$ supported on $E$ such that
		\begin{align}
		\nu([a,a+\delta])\leq \mathbf{\mathfrak{c}}_{7} \, \delta^{\beta_1},\label{Frostman 2}
		\end{align}
		for all $ a\in E$ and $\delta>0$, where $\mathbf{\mathfrak{c}}_{7}$ is a positive constant which depends on $\beta_1$ only. 
		Let us suppose that $\mathcal{C}_{d-\beta_1/H}(F)>0$, otherwise there is nothing to prove. It follows that for all $\gamma\in (0,\mathcal{C}_{d-\beta_1/H}(F))$ there is a probability measure $m$ supported on $F$ such that
		\begin{align}
		\mathcal{E}_{d-\beta_1/H}(m):=\int_F\int_F\frac{m(dx)m(dy)}{\|x-y\|^{d-\beta_1/H}}\leq \gamma^{-1}.\label{E,d}
		\end{align}
		Since $\nu\otimes m$ is a probability measure on $E\times F$, then applying   Fubini’s theorem and \eqref{estim Frostman} of Lemma \ref{estim Frostman lemma} we obtain 
		\begin{align}
		\mathcal{E}_{\rho_H,d}(\nu\otimes m)= \int_{E\times F} \int_{E\times F} \frac{\nu\otimes m(d u) \nu\otimes m(d v)}{(\rho_H(u,v))^d}\leq  \mathbf{\mathfrak{c}}_{4}\int_F\int_F\frac{m(dx)m(dy)}{\|x-y\|^{d-\beta_1/H}}\leq  \mathbf{\mathfrak{c}}_{4}\,\gamma^{-1}.\label{E,d,rho}
		\end{align}
		Consequently we have $\mathcal{C}_{\rho_H,\alpha}(E\times F)\geq  \mathbf{\mathfrak{c}}_{4}^{-1}\,\gamma$. Then we let  $\gamma\uparrow\mathcal{C}_{d-\beta_1/H}(F)$  to conclude that the first inequality in \eqref{estim capacity-measure} holds true.
		
		Now let us  prove the second inequality in \eqref{estim capacity-measure}. Let $l>\mathcal{H}^{d-\beta_2/H}(F)$ be arbitrary with $d-\beta_2/H>0$. Then there is a covering of $F$ by open balls $B(r_n)$ of radius $r_n$ such that 
		\begin{align}
		F \subset \bigcup_{n=1}^{\infty} B(r_{n}) \quad \text { and } \quad \sum_{n=1}^{\infty} (2r_{n})^{d-\beta_2/H} \leq l.\label{covering for F}
		\end{align}
		For all $n\geq 1$, let $E_{n,j}$, $j=1,...,N(E,2\,r_n^{1/H})$ be a family of open intervals of length $2\,r_n^{1/H}$ covering  $E$.  It follows that the family $E_{n,j}\times B(r_n),\,j=1,...,N(E,2\,r_n^{1/H})$, $n \geq  1$ gives a covering of $E\times F$ by open balls of radius $r_n$ for the parabolic metric $\rho_H$. 
		
		It follows from \eqref{upper Minkowski} that for all $\delta>0$  the number of open intervals of length $\delta$ needed to cover $E$ satisfies
		\begin{align}
		N(E,\delta)\leq \mathbf{\mathfrak{c}}_8\, \delta^{-\beta_2},\label{estim covering number}
		\end{align}
		where $\mathbf{\mathfrak{c}}_8$ is a positive and finite constant which depend on $E$ only. 
		Together with the estimates \eqref{covering for F} and \eqref{estim covering number} that have been established above, we have  
		\begin{align}
		\sum_{n=1}^{\infty} \sum_{j=1}^{N(E,2\, r_n^{1/H})}\left(2 r_{n}\right)^{d} \leq \mathbf{\mathfrak{c}}_{8}\,2^{-\beta_2(1-1/H)}\, \sum_{n=1}^{\infty} (2r_{n})^{d-\beta_2/H} \leq \mathbf{\mathfrak{c}}_{8}\,\,2^{-\beta_2(1-1/H)}\, l.\label{covering for E*F}
		\end{align}
		Then let $l\downarrow\mathcal{H}^{d-\beta_2/H}(F)$, the second inequality in \eqref{estim capacity-measure} follows with $\mathbf{\mathfrak{c}}_{6}=\mathbf{\mathfrak{c}}_{8}\,\,2^{-\beta_2(1-1/H)}$.
	\end{proof}
	
	It is well known that Hausdorff and Minkowski dimensions agree for many sets $E$. Often this is linked on the one hand to the geometric properties of the set, on the other hand it is a consequence of the existence of a sufficiently regular measure. Among the best known are Ahlfors-David regular sets defined as follows:
	\begin{itemize}
		\item[\textbf{(S)}]: Let $E\subset I$ and $\beta \in \left] 0,1\right]  $. We say that $E$ is $\beta$-regular if there exists a finite positive Borel measure $\nu$ supported on $E$ and positive constant $\mathbf{\mathfrak{c}}_{9}$ and  such that 
		\begin{equation} 
		\mathbf{\mathfrak{c}}_{9}^{-1}\,\delta^{\beta}\leq\nu([a-\delta,a+\delta])\leq\mathbf{\mathfrak{c}}_{9}\,\delta^{\beta}\,\,\text{ for all \ensuremath{a\in E}, \ensuremath{0<\delta\leq1}}.\label{condition S}
		\end{equation}
	\end{itemize}
	\begin{remark}
		1. If $E$ is the whole interval $I$ then $\beta$ in the condition \textbf{(S)} should be equal to $1$. This leads to the conclusion that the measure $\nu$ can be chosen as the normalized Lebesgue measure on $I$. In this case the above proposition  is simply Corollary 2.2 in \cite{Chen Xiao 2012}.
		
		2. The Cantor set $C(\lambda)$, $0<\lambda<1/2$, subset of $I$ with $\nu$ is the $\beta$-dimensional Hausdorff measure restricted to $C(\lambda)$ where $\beta=\dim C(\lambda)=\log(2)/\log(1/\lambda)$. For more details see Theorem 4.14 p.$67$  in Mattila \cite{Mattila}. In general, self similar subsets of $\mathbb{R}$ satisfying
		the open set condition are standard examples of regular sets, see \cite{Hutchinson}.
	\end{remark}

	According to Theorem 5.7 p.$80$ in \cite{Mattila}, for a set $E$ satisfying the condition \textbf{(S)} we have
	$$ \beta=\dim(E)=\underline{\dim}_M(E)=\overline{\dim}_M(E).$$
	In such case Proposition  \ref{corollary hitting} becomes 
	
	\begin{proposition}\label{corollary hitting 2}
		Let $B^H$, $f$ and $F$ as in Theorem \ref{main theorem hitting}. Let $E$ be a subset of $I$   satisfying  the condition \textbf{(S)}. Then there is a positive and finite constant $\mathbf{\mathfrak{c}}_{10}$ which depends on $E$, $F$, $H$, $K$ and $\beta$, such that 
		\begin{equation}
		\mathbf{\mathfrak{c}}_{10}^{-1}\, \mathcal{C}_{d-\beta/H}(F)\leq \mathbb{P}\{(B^H+f)(E)\cap F\neq \emptyset \}\leq 		\mathbf{\mathfrak{c}}_{10}\,\mathcal{H}^{d-\beta/H}(F).\label{upper-lower bound 2}
		\end{equation}
	\end{proposition}
\begin{proof}
Three cases are to be discussed here: (i) $\beta <Hd$, (ii) $\beta =Hd$ and (iii) $\beta >Hd$. Let us point out first that for the lower bound, the interesting cases are (i) and (ii) while for the upper bound it is the case (i) which requires proof. Indeed we have from \eqref{radial kernel} that $\mathcal{C}_{\alpha }(\cdot)=1$ for $\alpha <0$ and $\mathcal{H}^{\alpha}(.)$ is assumed to be equal to $1$ whenever $\alpha\leq 0$. In this regard, a close reading of the proof of Proposition \ref{corollary hitting} is required. Thus, we can clearly see that it is based on two key estimates, namely \eqref{Frostman 2} and \eqref{estim covering number} for the lower and upper bound respectively. In what follows we will establish such estimates under the condition  \textbf{(S)}. The estimation \eqref{Frostman 2} with $\beta$ is now a part of the condition  \textbf{(S)}. In order to establish  \eqref{estim covering number}, we will show that for all $\delta>0$  
\begin{align}
N(E,\delta)\leq C \delta^{-\beta},\label{estim covering number 2}
\end{align}
where $C$ is a positive and finite constant which depend on $I$ only. Indeed, let $0<\delta\leq 1$ and $P(E,\delta)$ be the greatest number of disjoint intervals $I_j$ centred in $x_{j}\in E$ with length $\delta$ required to cover $E$. Condition \textbf{(S)} ensures that $$\mathbf{\mathfrak{c}}_{9}^{-1}\, P(E,\delta)\, (\delta/2)^{\beta}\leq \sum_{j=1}^{P(E,\delta)}\nu(I_j)=\nu (E)\leq 1.$$ 
Using the fact that
$$N(E,2\delta)\leq P(E,\delta),$$
we obtain the desired estimation \eqref{estim covering number}. The rest of the proof follows closely the lines of that of  Proposition \ref{corollary hitting} especially given that Lemma \ref{estim Frostman lemma} takes into account the case (ii).
\end{proof}
	Following the same pattern as above we get the following proposition, which can be considered also as a corollary of Theorem 2.1 in \cite{Chen Xiao 2012}, for the subset $E$ of $I$  satisfying  the condition \textbf{(S)}.
	\begin{proposition}\label{corollary hitting CX12}
		Let $B^H$ and $F$ as in Theorem \ref{main theorem hitting}. Let $E$ be a subset of $I$   satisfying  the condition \textbf{(S)}. Then there is a positive and finite constant $\mathbf{\mathfrak{c}}_{11}$ which depends on $E$, $F$, $H$, $K$ and $\beta$, such that 
		\begin{equation}\label{upper-lower boundCX12}
		\mathbf{\mathfrak{c}}_{11}^{-1}\, \mathcal{C}_{d-\beta/H}(F)\leq \mathbb{P}\{B^H(E)\cap F\neq \emptyset \}\leq 		\mathbf{\mathfrak{c}}_{11}\,\mathcal{H}^{d-\beta/H}(F).
		\end{equation}
	\end{proposition}
We would like to point out that, when the drift $f$ is $H$-Hölder continuous and $E$ satisfies the condition \textbf{(S)}, Propositions \ref{corollary hitting 2} and \ref{corollary hitting CX12} assert that the hitting probabilities of $(B^H+f)$ behave like the ones of  $B^H$ in the following sense 
	
	$\bullet$ if $\dim(F)<d-\beta/H$  then  $\mathbb{P}\{(B^H+f)(E)\cap F\neq \emptyset \}=\mathbb{P}\{B^H(E)\cap F\neq \emptyset \}=0$,
	
	$\bullet$ if $\dim(F)>d-\beta/H$  then  $\mathbb{P}\{(B^H+f)(E)\cap F\neq \emptyset \}>0$ and  $\mathbb{P}\{B^H(E)\cap F\neq \emptyset \}>0$.
	
	This brings us to the following question: when $\dim(F)<d-\beta/H$, is it possible to get a function, with smaller Hölder order than H, for which $\mathbb{P}\{(B^H+f)(E)\cap F\neq \emptyset \}>0$? 
	In order to address this question, we need to consider $\alpha<H$ such that $d-\beta/\alpha<\dim(F)$ and $(\Omega^{\prime},\mathcal{F}^{\prime},\mathbb{P}^{\prime})$ another probability space on which we define a fractional Brownian motion $B^{\alpha}$ with Hurst parameter $\alpha$. We will work with the mixed process $Z^{H,\alpha}$ defined on the probability space $(\Omega\times\Omega^{\prime},\mathcal{F}\times\mathcal{F}^{\prime},\mathbb{P}\otimes\mathbb{P}^{\prime})$ by 
	\begin{align}
	Z^{H,\alpha}(t,(\omega,\omega^{\prime}))=B^H(t,\omega)+B^{\alpha}(t,\omega^{\prime}) \text{ for all  } t\geq 0 \text{ and } (\omega,\omega^{\prime})\in \Omega\times\Omega^{\prime}.
	\end{align}
	It is easy to see that $Z^{H,\alpha}=(Z^{H,\alpha}_1,...,Z^{H,\alpha}_d)$ where $Z^{H,\alpha}_i$ are independent copies of a real valued Gaussian process $Z^{H,\alpha}_0$ on $(\Omega\times\Omega^{\prime},\mathcal{F}\times\mathcal{F}^{\prime},\mathbb{P}\otimes\mathbb{P}^{\prime})$ with stationary increment and the covariance function given by $$\widetilde{\mathbb{E}}(Z^{H,\alpha}_0(s)Z^{H,\alpha}_0(t))=\frac{1}{2}(s^{2H}+t^{2H}-|t-s|^{2H})+\frac{1}{2}(s^{2\alpha}+t^{2\alpha}-|t-s|^{2\alpha}),$$
	where $\widetilde{\mathbb{E}}$ denote the expectation under the probability $\mathbb{P}\otimes \mathbb{P}^{\prime}$. The following lemma is about the strong local nondeterminism property of the process $Z$.
	Let $I\subset (0,1]$, be a closed interval, then we have 
	
	\begin{lemma}
		The real-valued process $\{Z^{H,\alpha}_0(t): t\geq 0\}$ satisfy the following  
		
		1. For all $s,t \in [0,1]$,
		\begin{align}
		|s-t|^{2\alpha}\leq \widetilde{\mathbb{E}}\left(Z^{H,\alpha}_0(t)-Z^{H,\alpha}_0(s)\right)^{2}\leq 2|s-t|^{2\alpha}.\label{slnd 1}
		\end{align}
		
		2. There exists a positive constant $C$ depending on $\alpha$, $H$ and $I$ only, such that 
		\begin{equation}
		\operatorname{Var}\left(Z^{H,\alpha}_{0}(u) \mid Z^{H,\alpha}_{0}\left(t_{1}\right), \ldots, Z^{H,\alpha}_{0}\left(t_{n}\right)\right) \geq C\left[\min _{0 \leq k \leq n}\left|u-t_{k}\right|^{2\alpha}+\min _{0 \leq k \leq n}\left|u-t_{k}\right|^{2 H}\right],
		\end{equation}
		for all integers $n\geq 1$, all $u,t_1,...,t_n \in I$ and $t_0=0$.
		
		3. There exists a positive constant $C$ depending on $\alpha$, $H$ and $I$ only.
		Such that for any $t\in I$ and any $0<r\leq t$,
		\begin{equation}
		\operatorname{Var}\left(Z^{H,\alpha}_{0}(t) \mid \, Z^{H,\alpha}_{0}(s); \mid s-t \mid \geq r\right) \geq C r^{2 \alpha}
		\end{equation}
		
	\end{lemma}
	
	The proof of the lemma is that of Proposition 4.2.  in \cite{Erraoui Hakiki}. Now we are able to provide an answer to the  above question.

	\begin{theorem}Let $E$ be a compact set  satisfying  the condition \textbf{(S)} and $F\subset \mathbb{R}^d$ be a compact set such that $\dim(F)\leq d-\beta/H$. Then for all $\alpha< \dfrac{\beta}{d-\dim(F)}$ and for all $\varepsilon>0$ small enough such that $\mathcal{C}_{d-\beta/(\alpha+\varepsilon)}(F)>0$, there exists a $\alpha$-Hölder continuous function $f : [0,1]\rightarrow \mathbb{R}^d$, satisfying
		\begin{equation}\label{First est}
		\mathbb{P}\{(B^H+f)(E)\cap F\neq \emptyset \}\geq {\fontsize{15}{0} \selectfont \textsf{c}} _{1}\, \mathcal{C}_{d-\beta/(\alpha+\varepsilon)}(F),
		\end{equation}
		
		where ${\fontsize{15}{0} \selectfont \textsf{c}}_{1}$ is a positive constant which depends on $\alpha,\beta,H,d$ and $\varepsilon$ only.
	\end{theorem}
	\begin{proof}
		Let us set $\alpha^{\prime}:=\alpha+\varepsilon$ with $\varepsilon>0$ small enough, and consider the stochastic process $Z^{\alpha^{\prime},H}$ stated above. The previous lemma tells us that $Z^{\alpha^{\prime},H}$  satisfies the conditions $(C1)$ and $(C2)$ of Theorem 2.1 in \cite{Chen Xiao 2012}. So we have $$ {\fontsize{15}{0} \selectfont \textsf{c}}_2^{-1} \, \mathcal{C}_{\rho_{\alpha^{\prime}},d}(E\times F)\leq  \widetilde{\mathbb{P}}\{Z(E)\cap F\neq \emptyset \}\leq {\fontsize{15}{0} \selectfont \textsf{c}}_2\, \mathcal{C}_{\rho_{\alpha^{\prime}},d}(E\times F),$$ ${\fontsize{15}{0} \selectfont \textsf{c}}_2\geq 1$ is a finite constant which depends on $I$, $F$ and $H$ only. Since $E$ satisfies the condition \textbf{(S)} we use  \eqref{estim capacity-measure} to obtain $$ {\fontsize{15}{0} \selectfont \textsf{c}}_3 \, \mathcal{C}_{d-\beta/\alpha^{\prime}}(F)\leq \widetilde{\mathbb{P}}\{Z(E)\cap F\neq \emptyset \}\leq {\fontsize{15}{0} \selectfont \textsf{c}}_4 \,\mathcal{H}^{d-\beta/\alpha^{\prime}}(F) ,$$
		where ${\fontsize{15}{0} \selectfont \textsf{c}}_3$ and ${\fontsize{15}{0} \selectfont \textsf{c}}_4$ are positive constants which depends on $d$, $\beta$, and $\alpha^{\prime}$ only. We can choose $\varepsilon$ small enough such that $\mathcal{C}_{d-\beta/\alpha^{\prime}}(F)>0$. Therefore by Fubini's theorem we get
		\begin{equation*}
		\mathbb{E}^{\prime}\left(\mathbb{P}\left\{(B^H+B^{\alpha^{\prime}}(\omega^{\prime}))(E)\cap F\neq \emptyset \right\}-{\fontsize{15}{0} \selectfont \textsf{c}}_5\, \mathcal{C}_{d-\beta/\alpha^{\prime}}(F)\right)>0,
		\end{equation*}
		for some fixed positive constant ${\fontsize{15}{0} \selectfont \textsf{c}}_5\in (0,{\fontsize{15}{0} \selectfont \textsf{c}}_3)$. The above inequality lead to
		
		\begin{equation*}
		\mathbb{P}^{\prime}\left\{\mathbb{P}\left\{(B^H+B^{\alpha^{\prime}}(\omega^{\prime}))(E)\cap F\neq \emptyset \right\}-{\fontsize{15}{0} \selectfont \textsf{c}}_5\, \mathcal{C}_{d-\beta/\alpha^{\prime}}(F)>0\right\}>0.
		\end{equation*}
		
		We therefore choose the function $ f$ among the paths of $ B ^ {\alpha^{\prime}} $ satisfying the above. 
	\end{proof}
	\begin{remark}
The same reasoning should also apply to check for all $\alpha<\dfrac{\beta}{d-\dim(F)}$ and for all $\varepsilon>0$ small enough such that $\mathcal{H}_{d-\beta/(\alpha+\varepsilon)}(F)<+\infty$, there exists a $\alpha$-Hölder continuous function $f : [0,1]\rightarrow \mathbb{R}^d$, satisfying
	\begin{equation}\label{Second est}
	\mathbb{P}\{(B^H+f)(E)\cap F\neq \emptyset \}\leq 		{\fontsize{15}{0} \selectfont \textsf{c}}_{6}\,\mathcal{H}^{d-\beta /(\alpha+\varepsilon)}(F),
	\end{equation}
where ${\fontsize{15}{0} \selectfont \textsf{c}}_{6}$ is a positive constant which depends on $\alpha,\beta,H,d$ and $\varepsilon$ only.	 
	\end{remark}
\section{Hitting points}
Let's start with the following fact:

Since $f$ is $H$-Hölder continuous Proposition 2.7. in \cite{Erraoui Hakiki} ensures that $\dim_{\Psi, H}(Gr_E(f))=\dim(E)$ and then according to Theorem 1.2. in \cite{Peres&Sousi2016} we have 
	$$\dim(B^H+f)(E)=\frac{\dim(E)}{H}\wedge d.$$ Hence, if $H\,d<\dim(E)$ Theorem 3.2. in \cite{Erraoui Hakiki} implies that $\mathbb{E}(\lambda_d(B^H+f)(E))>0$, and a simple application of Fubini's theorem  leads to
	\begin{align}
	\lambda_d\{x\in \mathbb{R}^d: \mathbb{P}\{\exists t\in E: B^H(t)+f(t)=x \}>0\}>0\label{ss}.
	\end{align}	  
	 On the other hand we have $d<\dim(E)/H=\dim_{\rho_H}(E\times \{x\})$ for all $x\in \mathbb{R}^d$, from which follows that $\mathcal{C}_{\rho_H,d}(E\times \{x\})>0$. Using Theorem \ref{main theorem hitting} we conclude that $B^H+f$ restricted on $E$ hits all points with positive probability which is stronger than $\eqref{ss}$.

When the function $ f $ loses the Hölder property, thus one wonders what about the less smooth functions? Our goal is to shed some light on this question. In fact we will need some additional information about the set $Gr_E(f)$ in order to study the hitting probabilities points for the process $B^H+f$. 	
First we provide, for a simply measurable Borel function $f$, the lower and upper bounds of hitting probabilities of points in terms of the parabolic capacity of $Gr_E(f)$ of order $d$ and the $d$-dimensional parabolic Hausdorff measure of $Gr_E(f)$ respectively. It can be also seen as an extension of Theorem \ref{main theorem hitting} to a measurable drift $f$ and $F=\{x\}$.
	
	\begin{proposition}\label{prop hit point}
		Let $\{B^H(t): t\in [0,1]\}$ be a $d$-dimensional fractional Brownian motion with Hurst index $H\in (0,1)$. Let $f: [0,1]\rightarrow \mathbb{R}^d$ be a bounded Borel measurable function and let $E\subset (0,1]$ be a Borel set. Then for all $x\in \mathbb{R}^d$ there is a finite constant ${\fontsize{14}{0} \selectfont \textbf{c}}_1\geq 1$ such that
		
		\begin{align}\label{upper-lower hitting points}
{\fontsize{14}{0} \selectfont \textbf{c}}_1^{-1}\mathcal{C}_{\rho_H,d}(Gr_E(f))\leq \mathbb{P}\left\lbrace \exists t\in E : (B^H+f)(t)=x\right\rbrace\leq {\fontsize{14}{0} \selectfont \textbf{c}}_1\, \mathcal{H}_{\rho_H}^{d}(Gr_E(f)).
		\end{align}
	\end{proposition}

	\begin{proof}
		we will closely follow  the same steps as in the proof  of Theorem \ref{main theorem hitting}.  We start with the upper bound using again the covering argument. Choose an arbitrary constant  $\gamma>\mathcal{H}_{\rho_H}^d(Gr_E(f))$. Then there is a covering of $Gr_E(f)$ by balls $\{B_{\rho_H}((t_i,y_i),r_i), i\geq 1\}$ in $\mathbb{R}_+\times\mathbb{R}^d$ such that 
		\begin{equation}
		Gr_E(f)\subseteq \bigcup_{i=1}^{\infty}B_{\rho_H}((t_i,y_i),r_i)\quad  \text{ and }\quad  \sum_{i=1}^{\infty}(2r_i)^{d}\leq \gamma.\label{cover Gr_E(f)}
		\end{equation}
		It is easy to see that
		\begin{equation}\label{event covering 2}
		\left\{ \exists s\in E:(B^{H}+f)(s)=x\right\} \subseteq\bigcup_{i=1}^{\infty}\left\{ \exists\,\left(s,f(s)\right)\in\left(t_{i}-r_{i}^{1/H},t_{i}+r_{i}^{1/H}\right)\times B(y_{i},r_{i})\text{ s.t. }(B^{H}+f)(s)=x\right\}.
		\end{equation}
		Since for every fixed $i\geq 1$ we have 
		\begin{equation}\label{event inclu}
		\left\{ \exists\,\left(s,f(s)\right)\in\left(t_{i}-r_{i}^{1/H},t_{i}+r_{i}^{1/H}\right)\times B(y_{i},r_{i})\text{ s.t. }(B^{H}+f)(s)=x\right\} 
		\subseteq \left\{ \inf_{|s-t_i|^{H}<r_i }\|B^H(s)-x-y_i\|\leq r_i \right\},
		\end{equation}
		then we get from Lemma \ref{l4} that 
		\begin{align}
		\mathbb{P}\left\{ \exists\,\left(s,f(s)\right)\in\left(t_{i}-r_{i}^{1/H},t_{i}+r_{i}^{1/H}\right)\times B(y_{i},r_{i})\text{ s.t. }(B^{H}+f)(s)=x\right\}  & \leq\mathbb{P}\left\{ \inf_{|s-t_{i}|^{H}<r_{i}}\|B^{H}(s)-x-y_{i}\|\right\} \nonumber\\
		& \leq {\fontsize{14}{0} \selectfont \textbf{c}}_{2}\,r_{i}^{d},\label{proba estim 2}
		\end{align}
		where $C_{2}$ depends on $H$, $E$ and $f$ only. Combining \eqref{cover Gr_E(f)}, \eqref{event covering 2}, \eqref{event inclu} and \eqref{proba estim 2} we derive that $$\mathbb{P}\left\lbrace \exists s\in E : (B^H+f)(s)=x\right\rbrace\leq {\fontsize{14}{0} \selectfont \textbf{c}}_3\, \gamma,$$ where $C_3$ depends only on $I,H$ and $f$. Let  $\gamma\downarrow \mathcal{H}_{\rho}^{d}(Gr_E(f))$, the upper bound in \eqref{upper-lower hitting points} follows.
		
		The lower bound in\eqref{upper-lower hitting points} holds also from the second moment argument. We assume that $\mathcal{C}_{\rho_H,d}(Gr_E(f))>0$, then let $\sigma$ be a measure supported on $Gr_E(f)$ such that 
		\begin{align}\label{ener<cap}
		\mathcal{E}_{\rho_H,d}(\sigma)=\int_{Gr_E(f)}\int_{Gr_E(f)}\frac{d\sigma(s,f(s))d\sigma(t,f(t))}{\rho_H((s,f(s)),(t,f(t)))^{d}}\leq \frac{2}{\mathcal{C}_{\rho_H,d}(Gr_E(f))}.
		\end{align}
		Let $\nu$ be the measure on $E$ satisfying $\nu:=\sigma\circ P_1^{-1} $ where $P_1$ is the projection mapping on $E$, i.e. $P_1(s,f(s))=s$. For $n\geq 1$ we consider a family of random measures $\nu_n$ on $E$ defined by
		\begin{equation}
		\begin{aligned} \int_{E} g(s) \nu_{n}(d s) &=\int_{E} (2 \pi n)^{d / 2} \exp \left(-\frac{n\|B^H(s)+f(s)-x\|^2}{2}\right) g(s) \nu(d s),
		\end{aligned}\label{seq rand meas Gr}
		\end{equation}
		where $g$ is an arbitrary measurable function on $\mathbb{R}_+$. Following the same steps as in \eqref{esp energy} we obtain that there exists two positives constants ${\fontsize{14}{0} \selectfont \textbf{c}}_4$ and ${\fontsize{14}{0} \selectfont \textbf{c}}_5$ such that
		
		\begin{equation}\label{energ-mes1}
		\mathbb{E}(\|\nu_n\|)
		\geq\int_E\left(\frac{2\pi}{1+s^{2H}}\right)^{d/2}\exp\left(-\frac{\|x-f(s)\|^2}{2s^{2H}}\right)\nu(ds)\geq {\fontsize{14}{0} \selectfont \textbf{c}}_{4}>0,
		\end{equation}  
		and 
		
		\begin{align}
		\mathbb{E}\left(\|\nu_{n}\|^{2}\right)   & =\int_{E}\int_{E}\nu(ds)\nu(dt)\int_{\mathbb{R}^{2d}}e^{-i(\left\langle \xi,x-f(s)\right\rangle +\left\langle \eta,x-f(t)\right\rangle )} \nonumber\\
		& \times\exp(-\frac{1}{2}(\xi,\eta)(n^{-1}I_{2d}+Cov(B^{H}(s),B^{H}(t)))(\xi,\eta)^{T})d\xi d\eta \nonumber\\
		& \leqslant {\fontsize{14}{0} \selectfont \textbf{c}}_{5}\int_{Gr_{E}(f)}\int_{Gr_{E}(f)}\frac{d\sigma(s,f(s))d\sigma(t,f(t))}{(\max\{|t-s|^{H},\|f(t)-f(s)\|\})^{d}}={\fontsize{14}{0} \selectfont \textbf{c}}_{5}\,\mathcal{E}_{\rho_{H},d}(\sigma)<\infty.
		\end{align}
		where the last inequality  is a direct consequence of  Lemma \ref{l5}. Using once again the Paley-Zygmund inequality, we conclude that $(\nu_n)_{n\geq 1}$ admits a subsequence converging weakly to a finite measure $\bar{\mu}$  supported on the set $\{(s,x)\in E \times F : B^H(s)+f(s)=x\}$, positive with positive probability and also satisfying the moment estimates of \eqref{esp energy}. Hence we have
		\begin{align}
		\mathbb{P}\left(\exists s\in E: (B^H+f)(s)=x\right)\geq \mathbb{P}\left(\|\bar{\mu}\|>0\right)\geq \frac{\mathbb{E}(\|\bar{\mu}\|)^2}{\mathbb{E}(\|\bar{\mu}\|^2)}\geq \frac{{\fontsize{14}{0} \selectfont \textbf{c}}_4^2}{{\fontsize{14}{0} \selectfont \textbf{c}}_5\mathcal{E}_{\rho_H,d}(\sigma)}.
		\end{align}
		Combining this with \eqref{ener<cap} yields the lower bound in \eqref{upper-lower hitting points}. The proof is completed.
	\end{proof}
\begin{corollary}
Let $\{B^{H}(t) : t \in [0,1]\}$ be a $d$-dimensional fractional Brownian motion of Hurst index $H\in (0,1)$ and $E\subset [0,1]$ be a Borel set. Then for any $\alpha<\text{dim}(E)/d \wedge H$ there exists a $\alpha$-Hölder continuous function $f : [0,1] \rightarrow \mathbb{R}^d$ such that, for all $x\in \mathbb{R}^d$, we have

\begin{align}
\mathbb{P}\left\{\exists t\in E : (B^{H}+f)(t)=x\right\}>0.
\end{align}
 In other words, that the restriction of $f$ to $E$ is non-polar for $B^{H}$.

\end{corollary}
		\begin{proof}
	Let $\varepsilon>0$ such that $\alpha''=\alpha+\varepsilon<\text{dim}(E)/d \wedge H$ and $B^{\alpha''}$ be the fractional Brownian motion defined above. It is known from Theorem 2.9 and Corollary 2.11 in \cite{Erraoui Hakiki} that, for $\alpha''<\text{dim}(E)/d \wedge H$, we have  $$\dim_{\rho_H}(Gr_E(B^{\alpha''}))=\text{dim}(E)/\alpha'' \wedge \left(\text{dim}(E)/H+d(1-\alpha''/H)\right)>d,\,\mathbb{P}^{\prime} a.s.$$ Then for any fixed $x\in \mathbb{R}^d$, Proposition \ref{prop hit point} tells us that for $\mathbb{P}^{\prime}$ almost all $\omega^{\prime}\in \Omega^{\prime}$ there is a positive constant $C=C(\omega^{\prime})$ such that 
		\begin{align*}
			\mathbb{P}\left\{ \exists s\in E : (B^H+B^{\alpha}(\omega^{\prime}))(s)= x\right\}\geq C\, \mathcal{C}_{\rho_H,d}\left(Gr_E(B^{\alpha''}(\omega^{\prime}))\right)>0.
		\end{align*}
	Hence, if we choose $f$ to be one of the trajectories of  $B^{\alpha''}$, which is $\alpha$-Hölder continuous,
	we obtain
	\begin{align*}
		\dim_{\rho_H}(Gr_E(f))>d.\label{effect drift}
	\end{align*}
	Therefore for any fixed $x\in \mathbb{R}^d$, $\left\{ (B^H+f(s),  s\in E\right\}$ hits $x$ with positive probability. 
	\end{proof}
	\begin{remark}
		1. 
		2. We mention that the covering argument used to prove the upper bound in \eqref{upper-lower hitting points} can also serve to show that for any Borel set $F\subset\mathbb{R}^d$, there exists a positive finite constant ${\fontsize{14}{0} \selectfont \textbf{c}}$ such that 
		\begin{align}
		\mathbb{P}\left\lbrace (B^H+f)(E)\cap F\neq \emptyset \right\rbrace\leq {\fontsize{14}{0} \selectfont \textbf{c}}\, \mathcal{H}_{\widetilde{\rho}_H}^d(Gr_E(f)\times F).
		\end{align}
		Here $\mathcal{H}_{\widetilde{\rho}_H}^{\alpha}(\centerdot)$ is the $\alpha$-dimensional Hausdorff measure on the metric space $(\mathbb{R}_+\times\mathbb{R}^d\times\mathbb{R}^d,\widetilde{\rho}_H)$, where the $\widetilde{\rho}_H$ is defined by 
		\begin{align*}
		\widetilde{\rho}_H((s,x,u),(t,y,v)):=\max\{|t-s|^H,\|x-y\|,\|u-v\|\}.
		\end{align*}
		But we have difficulty in proving the lower band in terms of  $\mathcal{C}_{\widetilde{\rho}_H,d}(Gr_E(f)\times F)$ even when $F$ has some smooth structure.
	\end{remark}
	
Accordingly, in view of the foregoing, can we expect the same result for others functions with fewer restrictions? This impels us to consider the class of reverse $\alpha$-Hölder continuous functions whose definition is as follows
	\begin{definition}\label{Reverse}
		We say that  a continous function $f : [0,1] \rightarrow \mathbb{R}^d$ is reverse $\alpha$-Hölder continuous, for $0<\alpha<1$,  if there exists a constant $C>0$ such that  for any interval $J\subset [0,1]$, we have 
		\begin{align}
		\sup_{s,t\in J}\|f(s)-f(t)\|\geq C |J|^{\alpha},
		\end{align}
		where $|J|$ is the diameter of the interval $J$.
	\end{definition}
Recall that this notion is closely linked to the geometric properties of the graph of the function $f$. For $d=1$, a famous example of a function satisfying the reverse $\alpha$-Hölder continuity condition is the Weierstrass function given by
	\begin{align}\label{Weierstrass function}
	\mathcal{W}_{\tau,\theta}(t)=\sum_{n=0}^{\infty}\tau^{n}\cos\left( 2\pi \theta ^nt\right) , \quad t\in [0,1],
	\end{align}
	for $\tau <1<\theta $ and $\tau \, \theta>1$ and $\alpha=-\log(\tau)/\log(\theta)$. See \cite{Bishop Peres} or \cite{Hardy} 
	for the proof.
	Recently Shen \cite{Shen}, improving result of Barański, Bárány and Romanowska \cite{Baranski Barany Romanowska}, proved that for any integer
	$\theta\geq2$ and any $\tau\in(\theta^{-1},1)$, the Hausdorff dimension
	of the graph of the Weierstrass function $\mathcal{W}_{\tau,\theta}$
	is equal to $2+\dfrac{\log\left(\tau\right)}{\log\left(\theta\right)}$.
		It is worth pointing out that, for $H=1$, the metric $\rho_{H}$ on $\mathbb{R}_{+}\times\mathbb{R}$
	defined by
	\[
	\rho_{H}((s,x),(t,y))=\max\{|t-s|^{H},\vert x-y\vert\}\quad\forall(s,x),(t,y)\in\mathbb{R}_{+}\times\mathbb{R}
	\]
	is nothing but the metric derived from the Maximum norm on $\mathbb{R}^{2}$.
	Now using a comparison result for the Hausdorff parabolic dimensions
	with different parameters, see Proposition 2.5 in  Erraoui and Hakiki \cite{Erraoui Hakiki} (which
	remains valid also for $H$ and $1$), one can check that
	\[
	1<2+\dfrac{\log\left(\tau\right)}{\log\left(\theta\right)}=\dim(Gr_{[0,1]}(\mathcal{W}_{\tau,\theta}))\leq\dim_{\rho_{H}}(Gr_{[0,1]}(\mathcal{W}_{\tau,\theta})).
	\]
	Thus $\dim_{\rho_{H}}(Gr_{[0,1]}(\mathcal{W}_{\tau,\theta}))>1$ and therfore $\mathcal{C}_{\rho_H,1}(Gr_{[0,1]}(\mathcal{W}_{\tau,\theta})))>0$. This is expressed in the following 
		\begin{proposition}\label{polarity Weierstrass function}
		For any integer $\theta \geq  2$ and  any $ \tau \in  (\theta^{-1} , 1)$, the Weierstrass function $\{\mathcal{W}_{\tau,\theta}(t),\,t\in [0,1]\}$ is non-polar for real valued fractional Brownian motion $\{B^H(t),\,t\in [0,1]\}$ with Hurst index $H\in (0,1)$.
	\end{proposition}
\begin{remark}
It is worth mentioning that the Hölder continuity of order $\alpha=-\log(\tau)/\log(\theta)$ is also met by the Weierstrass function $\{\mathcal{W}_{\tau,\theta}(t),\,t\in [0,1]\}$, cf. Lemma 5.1.8 in \cite{Bishop Peres}.
\end{remark}

	It is also interesting to note that, in the same context, Theorem 4 in \cite{Przytycki Urbanski} affirms that, for any $\alpha$-Hölder and reverse $\alpha$-Hölder continuous function $f$ with $\alpha \in (0,1)$, we have $\dim_{[0,1]}(Gr(f))>1$ which was the key element in the proof of the above proposition. Therefore, we will come to the same conclusion as the one for the Weierstrass function stated as follows
	\begin{proposition}\label{polarity General Weierstrass function}
		Suppose that $0<\alpha <1 $ and $f : [0,1] \rightarrow \mathbb{R}$ is $\alpha$-Hölder and reverse $\alpha$-Hölder continuous function. Then $f$ is non-polar for real valued fractional Brownian motion $\{B^H(t),\,t\in [0,1]\}$ with Hurst index $H\in (0,1)$.
	\end{proposition}
	 
	This raises the question whether the result remains valid in higher dimensions. However Proposition \ref{polarity General Weierstrass function} cannot be extended to $d$-dimensional case as is shown by the following 
	\begin{proposition}\label{contre exemple}
		Let $\{B^{H}(t) : t \in [0,1]\}$ be a $d$-dimensional fractional Brownian motion of Hurst index $H\in (0,1)$ such that $H\,d>1$. Then there exists a $\alpha$-Hölder  and reverse $\alpha$-Hölder continuous function $f :[0,1] \rightarrow \mathbb{R}^d$ such that the process $\{(B^{H}+f)(t) : t \in [0,1]\}$ does not hit $x$ for all $x\in \mathbb{R}^d$, i.e. $$\mathbb{P}\left\lbrace \exists t \in [0,1] : (B^H+f)(t)=x\right\rbrace=0.$$
	\end{proposition}
	\begin{proof} Let $d\geq 2$ and define the function $f$ from $[0,1]$ to $\mathbb{R}^{d}$ by : $f(t)=\left(\mathcal{W}_{\tau,\theta}(t),\cdots,\mathcal{W}_{\tau,\theta}(t)\right)$ where $\mathcal{W}_{\tau,\theta}$ is the Weierstrass function given in \eqref{Weierstrass function}.
		It is easy to see that $f$ satisfies the $\alpha$-Hölder and reverse $\alpha$-Hölder
		conditions with $\alpha=-\dfrac{\log\left(\tau\right)}{\log\left(\theta\right)}$. Moreover we have 
 $\dim_{\rho_{H}}(Gr_{[0,1]}(f))=\dim_{\rho_{H}}(Gr_{[0,1]}(\mathcal{W}_{\tau,\theta}))$.
	It follows from Proposition 2.5 in  Erraoui and Hakiki \cite{Erraoui Hakiki} that $$\dim_{\rho_{H}}(Gr_{[0,1]}(f))\leq\dim(Gr_{[0,1]}(f))-1+\frac{1}{H}=1+\dfrac{\log\left(\tau\right)}{\log\left(\theta\right)}+\frac{1}{H}.$$
	Then for any $\theta >1$ and $\tau \in (\theta^{-1},\theta^{d-(1/H)-1}\wedge 1)$ we have $\dim_{\rho_{H}}(Gr_{[0,1]}(f))<d$.	The later entails that $\mathcal{H}_{\rho_{H}}^{d}(Gr_{[0,1]}(f))=0$. Proposition
		\eqref{prop hit point} will allow us to achieve the
		desired outcome. 
\end{proof}	

\bibliographystyle{abbrv}

\end{document}